\documentclass[10pt]{amsart}
\usepackage{amssymb, enumitem}
\usepackage[all]{xy}
\usepackage{hyperref, aliascnt}
\usepackage{mathtools}
\usepackage{bbm}
\usepackage[english]{babel}

\def\today{\number\day\space\ifcase\month\or   January\or February\or
   March\or April\or May\or June\or   July\or August\or September\or
   October\or November\or December\fi\   \number\year}

\theoremstyle{definition}
\newtheorem{lma}{Lemma}[section]

\newaliascnt{thmCt}{lma}
\newtheorem{thm}[thmCt]{Theorem}
\aliascntresetthe{thmCt}

\newaliascnt{corCt}{lma}
\newtheorem{cor}[corCt]{Corollary}
\aliascntresetthe{corCt}

\newaliascnt{propCt}{lma}
\newtheorem{prop}[propCt]{Proposition}
\aliascntresetthe{propCt}

\newtheorem*{thm*}{Theorem}
\newtheorem*{cor*}{Corollary}
\newtheorem*{prop*}{Proposition}

\newaliascnt{pgrCt}{lma}

\aliascntresetthe{pgrCt}

\newaliascnt{dfCt}{lma}
\newtheorem{df}[dfCt]{Definition}
\aliascntresetthe{dfCt}

\newaliascnt{remCt}{lma}
\newtheorem{rem}[remCt]{Remark}
\aliascntresetthe{remCt}

\newaliascnt{remsCt}{lma}

\aliascntresetthe{remsCt}

\newaliascnt{egCt}{lma}
\newtheorem{eg}[egCt]{Example}
\aliascntresetthe{egCt}

\newaliascnt{egsCt}{lma}

\aliascntresetthe{egsCt}

\newaliascnt{qstCt}{lma}

\aliascntresetthe{qstCt}

\newaliascnt{pbmCt}{lma}

\aliascntresetthe{pbmCt}

\newaliascnt{notaCt}{lma}
\newtheorem{nota}[notaCt]{Notation}
\aliascntresetthe{notaCt}

\newcommand{\beq}{\begin{equation}}
\newcommand{\eeq}{\end{equation}}
\newcommand{\beqa}{\begin{eqnarray*}}
\newcommand{\eeqa}{\end{eqnarray*}}
\newcommand{\bal}{\begin{align*}}
\newcommand{\eal}{\end{align*}}
\newcommand{\bi}{\begin{itemize}}
\newcommand{\ei}{\end{itemize}}
\newcommand{\be}{\begin{enumerate}}
\newcommand{\ee}{\end{enumerate}}

\newcommand{\af}{\alpha}
\newcommand{\bt}{\beta}

\newcommand{\dt}{\delta}
\newcommand{\ep}{\varepsilon}
\newcommand{\zt}{\zeta}

\newcommand{\Q}{{\mathbb{Q}}}

\newcommand{\Z}{{\mathbb{Z}}}
\newcommand{\R}{{\mathbb{R}}}
\newcommand{\C}{{\mathbb{C}}}
\newcommand{\N}{{\mathbb{N}}}

\newcommand{\K}{{\mathcal{K}}}

\newcommand{\U}{{\mathcal{U}}}

\newcommand{\T}{{\mathbb{T}}}
\newcommand{\D}{{\mathcal{D}}}

\newcommand{\Ot}{{\mathcal{O}_2}}
\newcommand{\OI}{{\mathcal{O}_{\I}}}

\newcommand{\id}{{\mathrm{id}}}

\newcommand{\diag}{{\mathrm{diag}}}

\newcommand{\card}{{\mathrm{card}}}
\newcommand{\Aut}{{\mathrm{Aut}}}

\newcommand{\Ad}{{\mathrm{Ad}}}

\newcommand{\dirlim}{\varinjlim}

\newcommand{\tfae}{the following are equivalent}
\newcommand{\ifo}{if and only if }

\newcommand{\ca}{$C^*$-algebra}
\newcommand{\cas}{$C^*$-algebras}
\newcommand{\uca}{unital $C^*$-algebra}

\newcommand{\hm}{homomorphism}

\newcommand{\ATa}{A$\T$-algebra}

\newcommand{\cp}{crossed product}

\newcommand{\Rp}{Rokhlin property}

\newcommand{\I}{\infty}

\title[]{Circle actions on UHF-absorbing $C^*$-algebras}

\author{Eusebio Gardella}

\date{\today}

\thanks{This material is based upon work supported by the
  US National Science Foundation through the author's thesis advisor's Grant
DMS-1101742. The financial support is gratefully acknowledged.}

\address{Department of Mathematics, University  of Oregon,
      Eugene OR 97403-1222, USA.}

\email[]{gardella@uoregon.edu}

\begin{document}

\begin{abstract}
We study circle actions with the Rokhlin property, in relation to their restrictions to finite subgroups.
We construct examples showing the following: the restriction of a circle action with the Rokhlin property
(even on a real rank zero $C^*$-algebra), need not have the Rokhlin property; and even if every restriction
of a given circle action has the Rokhlin property, the circle action itself need not have it. As a positive
result, we show that the restriction of a circle action with the Rokhlin property to the subgroup
$\mathbb{Z}_n$ has the Rokhlin property if the underlying algebra absorbs $M_{n^\infty}$. The condition on
the algebra is also necessary in most cases of interest.

Despite the fact that there are no circle actions with the Rokhlin property on UHF-algebras, we construct
many such actions on certain UHF-absorbing simple A$\T$-algebras. Additionally, we show that circle actions
with the Rokhlin property on $\mathcal{O}_2$-absorbing $C^*$-algebras are generic, in a suitable sense.
\end{abstract}

\maketitle

\tableofcontents

\section{Introduction}

The interplay between \ca s and dynamics has a long and rich history. Crossed products have provided some of the most interesting examples of
\ca s. Some algebraic properties are preserved under formation of crossed products in great generality. For example, crossed products of type
I $C^*$-algebras by compact groups are type I, and crossed products of nuclear $C^*$-algebras by amenable groups are nuclear, regardless of the
action. On the other hand, for preservation of other (usually stronger) properties, one must assume some kind of freeness condition on the action.
This is best seen in the commutative setting, where the Atiyah-Segal completion theorem (specifically as in the statement of Theorem 1.1.1 in \cite{Phi_Book}) shows how free actions on compact
spaces enjoy a number of nice analytic and algebraic properties.\\
\indent In the noncommutative setting, there are several different notions of freeness for actions, and many of them are surveyed in
\cite{Phi_survey}. For finite groups, and in roughly decreasing
order of strength, there are: the \Rp\ (see \cite{Izu_RpI}), the tracial
 Rokhlin property (see \cite{Phi_tracialFirst}), pointwise outerness, and hereditary saturation
(see \cite{Phi_Book}), just to mention a few. Among these, the \Rp\ is the strongest one, and is therefore less common than
the other notions of freeness. For example, the Rokhlin property for finite groups implies that the underlying algebra has non-trivial projections, ruling out the existence of such actions on many
$C^*$-algebras of interest, such as the Jiang-Su algebra $\mathcal{Z}$. There are also less obvious $K$-theoretic obstructions
to the \Rp. See Theorem 3.13 in \cite{Izu_RpI}. On the other hand, the \Rp\ implies very strong structure preservation results
for crossed products (see Theorem 2.3 in \cite{Phi_survey} for a list of properties that are preserved by Rokhlin actions, and
see \cite{HirWin_Rp}, \cite{OsaPhi_CPRP} and \cite{Phi_tracialFirst} for the proofs of most of them), and it is the hypothesis
in most theorems on classification of group actions (see Theorems 3.4 and 3.5 in \cite{Izu_RpII}, and see Theorem 4.7 in
\cite{Gar_Kir1}). These have been the main uses
of the \Rp: obtaining structure results for the \cp, and classification of actions.\\
\indent Besides finite groups, the \Rp\ has been extensively studied for automorphisms (see \cite{HerOcn_StabInteg} and
\cite{Kis_AutUHF}) and flows (see \cite{Kis_flows}). In \cite{HirWin_Rp}, Hirshberg and Winter introduced the \Rp\
for an action of a second-countable compact group on a unital \ca, and they proved that absorption of a strongly self-absorbing \ca\ and approximate
divisibility pass to crossed products by such actions. \\
\indent It is natural to try to generalize the results on the structure of the \cp\ in \cite{OsaPhi_CPRP} to arbitrary compact groups. This will
be done in \cite{Gar_CptRok}. Another natural direction is to explore the classification of Rokhlin actions of compact groups on
certain classes of classifiable \cas, generalizing or at least complementing Izumi's work for finite group actions with the \Rp. As a first step in this direction, we study Rokhlin actions of the circle on \ca s that absorb a UHF-algebra,
specifically in relation to their restrictions to finite subgroups. One of our main results, \autoref{characterize restriction has Rp}, asserts that under suitable assumptions on the UHF-algebra, all such restrictions have the
Rokhlin property. This fact, together with Izumi's classification of finite group actions with the Rokhlin
property, will be used in subsequent work to classify circle actions on UHF-absorbing \ca s.

Similarly to what happens with finite groups, Rokhlin actions of compact groups are rare, and there are \ca s that do not have any
action of any non-trivial compact group with the \Rp, such as the Cuntz algebra $\OI$ or the Jiang-Su algebra $\mathcal{Z}$. In this sense, $C^*$-algebras that
absorb $\Ot$ form a distinguished class since they have many actions
with the \Rp. In fact, the set of circle actions with the \Rp\ on a separable $\Ot$-absorbing \ca\ is a dense $G_\delta$-set in the space of all
circle actions. See \autoref{rokhlin action are a Gdelta set}.

\vspace{0.3cm}

This paper is organized as follows. We establish the notation that will be used throughout in the rest of the Introduction.
In Section 2, we introduce the definition of the Rokhlin property for circle actions on \uca, and derive
some of its basic properties which will be frequently used in the later sections. We also provide a number of examples of circle actions on
\ca s with the \Rp, mostly on simple \ca s. In contrast, we show in \autoref{no direct limit actions with the Rp on AF-algebras} that
no direct limit action of the circle on a UHF-algebra can have the \Rp. In Subsection 2.2, we specialize to circle actions on $C^*$-algebras
that absorb the Cuntz algebra $\Ot$, a class of \ca s which is special from the point of view of the Rokhlin property.
We show in \autoref{rokhlin action are a Gdelta set} that circle actions with the \Rp\ are
generic on separable, unital, $\Ot$-absorbing \ca s, a fact that
should be contrasted with \autoref{no direct limit actions with the Rp on AF-algebras}.\\
\indent Section 3 is devoted to showing the following: if $A$ is a separable, \uca\ that absorbs the UHF-algebra $M_{n^\I}$, and if $\alpha
\colon \T\to\Aut(A)$ is an action with the \Rp, then the restriction of $\alpha$ to the finite cyclic group $\Z_n\subseteq \T$ has the \Rp.
See \autoref{characterize restriction has Rp}. The condition that $A$ absorb $M_{n^\I}$ is shown to be necessary in most cases of interest;
see \autoref{converse to restriction has Rp if algebra absorbs UHF}. We also give examples of circle actions with the Rokhlin property
such that \emph{no} restriction to any finite cyclic group has the \Rp; see \autoref{eg restriction on S1} and
\autoref{eg restriction on pi ca}. Additionally, \autoref{all restrictions have Rp, but action does not}
and \autoref{all restrictions have Rp, but action does not on pi} show that even if a circle action has the
property that \emph{every} restriction to a finite subgroup has the Rokhlin property, the action itself need
not have the Rokhlin property, even on Kirchberg algebras satisfying the UCT.

\textbf{Acknowledgements.} The author is grateful to his advisor Chris Phillips for a number of helpful
conversations, as well as for feedback on an earlier draft of this work. He also wishes to thank the annonymous referee
for a number of suggestions that improved the quality of the present paper. 

\subsection{Notation and preliminaries}
We adopt the convention that $\{0\}$ is not a \uca, that is, we require that $1\neq 0$ in a \uca. We take $\N=\{1,2,\ldots\}$. For a separable \ca\ $A$, we denote by $\Aut(A)$ the automorphism group of $A$, which is equipped with the topology of pointwise norm convergence. In this topology, a sequence $(\varphi_n)_{n\in\N}$ converges to $\varphi\in\Aut(A)$ if and only if for every $\varepsilon>0$ and every compact set $F\subseteq A$, there exists $m\in\N$ such that $\|\varphi_m(a)-\varphi(a)\|<\varepsilon$ for all $a\in F$.\\
\indent If $A$ is moreover unital, then $\U(A)$ denotes the unitary group of $A$, and two automorphisms $\varphi$ and $\psi$ of $A$ are said to be approximately unitarily equivalent if $\varphi\circ\psi^{-1}$ is approximately inner.\\
\indent For a locally compact group $G$, an action of $G$ on a $C^*$-algebra $A$ is always assumed to be a \emph{continuous} group homomorphism from $G$ into $\Aut(A)$, unless otherwise stated. If $\alpha\colon G\to\Aut(A)$ is an action of $G$ on $A$, then we will denote by $A^\alpha$ the fixed-point subalgebra of $A$ under $\alpha$. \\
\indent For a \ca\ $A$, we set
\begin{align*} \ell^\I(\N,A)&=\left\{(a_n)_{n\in\N}\in A^{\N}\colon \sup_{n\in\N}\|a_n\|<\I\right\};\\
c_0(\N,A)&=\left\{(a_n)_{n\in\N}\in\ell^\I(\N,A)\colon \lim\limits_{n\to\I}\|a_n\|= 0\right\};\\
A_\I&= \ell^\I(\N,A)/c_0(\N,A).\end{align*}
We identify $A$ with the constant sequences in $\ell^\I(\N,A)$ and with their image in $A_\I$. We write $A_\I\cap A'$ for the central sequence algebra of $A$, that is, the relative commutant of $A\in A_\I$. For a bounded sequence $(a_n)_{n\in\N}\in A$, we denote by $\overline{(a_n)}_{n\in\N}$ its image in $A_\I$. We have
$$A_\I\cap A'=\left\{\overline{(a_n)}_{n\in\N}\in A_\I \colon \lim\limits_{n\to\I}\|a_na-aa_n\|= 0 \mbox{ for all } a\in A\right\}.$$
\indent If $\alpha\colon G\to\Aut(A)$ is an action of $G$ on $A$, then there are actions of $G$ on $A_\I$ and on $A_\I\cap A'$, both denoted by $\alpha_\I$. Note that unless the group $G$ is discrete, these actions will not be continuous in general.\\
\indent Given $n\in\{2,3,\ldots,\I\}$, we denote by $\mathcal{O}_n$ the Cuntz algebra with canonical generators $\{s_j\}_{j=1}^n$ satisfying the usual relations (see for example Section 4.2 in \cite{Ror_BookClassif}).\\
\indent We denote the circle group by $\T$, and identify it with the set of complex numbers of modulus 1. The finite cyclic group of order $n$ will be denoted by $\Z_n$, and we will usually identify $\Z_n$ with the $n$-th roots of unity in $\T$, and in this fashion we will regard $\Z_n$ as a subgroup of the circle.

\section{The Rokhlin property: rigidity and genericity}

We begin this section by recalling the definition of the \Rp\ for a finite group action on a \uca.

\begin{df}\label{RP} Let $A$ be a \uca, let $G$ be a finite group, and let $\alpha\colon G\to\Aut(A)$ be an action.
We say that $\alpha$ has the \emph{\Rp} if for every $\varepsilon>0$ and for every finite set $F\subseteq A$ there exist orthogonal
projections $e_g\in A$ for $g\in G$ such that
\be\item $\|\alpha_g(e_h)-e_{gh}\|<\varepsilon$ for all $g$ and $h\in G$
\item $\|e_ga-ae_g\|<\varepsilon$ for all $g\in G$ and all $a\in F$
\item $\sum\limits_{g\in G}e_g=1$.\ee\end{df}

The definition of the \Rp\ for finite group actions on \ca s was originally introduced by Izumi in \cite{Izu_RpI}, although a
similar
notion has been studied by Herman and Jones in \cite{HerJon_period} for $\Z_2$ actions on UHF-algebras, and by Herman and Ocneanu in
\cite{HerOcn_StabInteg} for integer actions. The Rokhlin property also played a crucial role in the classification
of finite group actions on von Neumann algebras.\\
\indent The following is part of Proposition 2.14 in \cite{Phi_survey}, and we include the proof for the convenience of the reader.
This result should be compared with \autoref{eg restriction on S1} and \autoref{eg restriction on pi ca}.

\begin{prop}\label{finite gp Rp properties} Let $A$ be a \uca, let $G$ be a finite group, and let $\alpha\colon G\to\Aut(A)$ be an action with the \Rp.
If $H\subseteq G$ is a subgroup, then $\alpha|_H$ has the \Rp.\end{prop}
\begin{proof} Set $n=\card(G/H)$. Given $\varepsilon>0$ and a finite subset $F\subseteq A$, choose projections $e_g$ for $g\in G$ as in the
definition of the \Rp\ for $F$ and $\frac{\ep}{n}$. We claim that the projections $f_h=\sum\limits_{\overline{x}\in G/H} e_{hx}$ for $h\in H$, form a
family of Rokhlin projections for the action $\alpha|_H$, the finite set $F$ and tolerance $\varepsilon$. \\
\indent Given $h$ and $k\in H$, we have
\begin{align*} \|\alpha_k(f_h)-f_{kh}\|&=\left\|\sum\limits_{\overline{x}\in G/H} \alpha_k(e_{hx})-e_{khx}\right\|\\
 &\leq \sum\limits_{\overline{x}\in G/H}\left\|  \alpha_k(e_{hx})-e_{khx}\right\| \leq \mbox{card}(G/H)\frac{\ep}{n}=\ep.
\end{align*}
Finally, for $a\in F$ and $h\in H$, we have
\[\|af_h-f_ha\|\leq \sum\limits_{\overline{x}\in G/H}\left\| ae_{hx}-e_{hx}a\right\|< \ep.\]\end{proof}

Hirshberg and Winter defined the \Rp\ for an arbitrary action of a compact, second countable group in \cite{HirWin_Rp}. In the case of
the circle, and using semiprojectivity of $C(\T)$, one can show that their definition is equivalent to the following.

\begin{df}\label{def RP} Let $A$ be a unital \ca\ and let $\alpha\colon\T\to\Aut(A)$ be a continuous action. Then $\alpha$ is said to have the \emph{Rokhlin property} if for every finite subset $F\subseteq A$ and every $\varepsilon>0$, there exists a unitary $u\in\U(A)$ such that
\be\item $\|\alpha_\zeta(u)-\zeta u\|<\varepsilon$ for all $\zeta\in \T$ and
\item $\|ua-au\|<\varepsilon$ for all $a\in F$.\ee\end{df}

\begin{rem} In order to check condition (2) in \autoref{def RP}, it is enough to consider finite subsets of any set of
 generators of $A$. It is also immediate to show that if in
\autoref{def RP} one allows the set $F$ to be compact instead of finite, one obtains an equivalent definition.
These easy observations will be used repeatedly and without reference.\end{rem}

We present some basic properties of circle actions with the Rokhlin property, some of which resemble those of free actions on spaces. For instance, the
proposition below is the analog of the fact that a diagonal action on a product space is free if one of the factors is free.

\begin{prop}\label{Properties RP}
Let $A$ be a \uca, and let $\af \colon \T \to \Aut (A)$ have the \Rp. If $\bt \colon \T \to \Aut (B)$ is any action of $\T$ on a \uca,
then the tensor product action $\zeta \mapsto \af_\zeta \otimes \bt_\zeta$ of $\T$ on $\Aut (A \otimes B),$ for any $C^*$-tensor product
on which it is defined, has the \Rp.\end{prop}
\begin{proof} Let $\varepsilon>0$ and let $F'\subseteq A$ and $F''\subseteq B$ be finite subsets of the respective unit balls of $A$ and $B$.
Set
$$F=\{a\otimes b\colon a\in F',b\in F''\},$$
which is a finite subset of $A\otimes B$. Using the \Rp\ for $\alpha$, choose a unitary $u\in A$ such that the conditions in \autoref{def RP}
are satisfied for $\ep$ and $F'$.
Set $v=u\otimes 1\in \U(A\otimes B)$. For $x=a\otimes b\in F$, we have
$$\|vx-xv\|=\|(ua-au)\otimes b\|\leq \|ua-au\|\|b\|<\ep.$$
On the other hand,
\begin{align*}\|(\alpha\otimes\beta)_\zeta(v)-\zeta v\|&=\|(\alpha_\zeta(u)\otimes \beta_\zeta(1))-\zeta (u\otimes 1)\|=\|\alpha_\zeta(u)-\zeta u\|<\varepsilon,\end{align*}
for all $\zeta\in\T$, which finishes the proof. \end{proof}

We point out that a tensor product action may have the \Rp\ without any of the tensor factors having the \Rp, even if one of the actions
is the trivial action. This is analogous to
the fact that a diagonal action on a product
space may be free without any of the factors being free, except that such examples with the trivial action as one of the factors do not exist in
the commutative setting.\\
\indent The proposition below is the analog of the fact that an equivariant inverse limit of free actions is free.

\begin{prop} \label{direct limit and Rp} If $A = \varinjlim (A_n,\iota_n)$ is a direct limit of \ca s with unital maps, and
$\af \colon \T \to \Aut (A)$ is an action obtained as the direct limit of actions $\af^{(n)} \colon \T \to \Aut (A_n),$ such that
$\af^{(n)}$ has the \Rp\  for all $n,$ then $\af$ has the \Rp. \end{prop}
\begin{proof} Let $F\subseteq A$ be a finite set, and let $\varepsilon>0$. Write $F=\{a_1,\ldots,a_N\}$. Since
$\bigcup\limits_{n\in\N} \iota_{n,\I}(A_n)$ is dense in $A$, there exist $n\in\N$ and $F'=\{b_1,\ldots,b_N\}\subseteq A_n$ such that
$\|a_j-\iota_{n,\I}(b_j)\|<\frac{\ep}{3}$ for $j=1,\ldots,N$. Since $\af^{(n)}$ has the \Rp, there exists a unitary $u\in A_n$ such
that $\left\|\af^{(n)}_\zeta(u)-\zeta u\right\|<\frac{\ep}{3}$ for all $\zeta\in \T$ and $\|b_ju-ub_j\|<\frac{\ep}{3}$ for all $j=1,\ldots,N$.
Notice that $\iota_{n,\I}(u)$ is a unitary in $A$, since the connecting maps are unital. Moreover, if $\zeta\in\T$, then
$$\|\af_\zeta(\iota_{n,\I}(u))-\zeta \iota_{n,\I}(u)\|=\left\|\iota_{n,\I}(\alpha^{(n)}_\zt(u))-\iota_{n,\I}(\zeta u)\right\| <\frac{\ep}{3}<\varepsilon.$$
Finally,
\begin{align*} &\|\iota_{n,\I}(u)a_j-a_j\iota_{n,\I}(u)\|\\
 &\ \ \ \ \ \ \leq \|\iota_{n,\I}(u)a_j-\iota_{n,\I}(u)\iota_{n,\I}(b_j)\|+\|\iota_{n,\I}(u)\iota_{n,\I}(b_j)-\iota_{n,\I}(b_j)\iota_{n,\I}(u)\|\\
&\ \ \ \ \ \  \ \ \ \ \ \ \ \ \ \ \ \ \ \ +\|\iota_{n,\I}(b_j)\iota_{n,\I}(u)-a_j\iota_{n,\I}(u)\|\\
&\ \ \ \ \ \ \leq\|a_j-\iota_{n,\I}(b_j)\|+\|ub_j-b_ju\|+\|\iota_{n,\I}(b_j)-a_j\|\\
&\ \ \ \ \ \ <\frac{\ep}{3}+\frac{\ep}{3}+\frac{\ep}{3}=\varepsilon.\end{align*}
Hence $\iota_{n,\I}(u)$ is the desired unitary for $F$ and $\varepsilon$, and thus $\alpha$ has the \Rp. \end{proof}

We have the following convenient result, which turns out to be crucial in some proofs, in particular in the classification of
Rokhlin actions of the circle on Kirchberg algebras; see \cite{Gar_Kir1} and \cite{Gar_Kir2}. In the present
work, we will use \autoref{can replace estimate by equality} in the proof of \autoref{W(F,epsilon) is dense}.

\begin{prop}\label{can replace estimate by equality} Let $A$ be a \uca, let $\alpha\colon \T\to\Aut(A)$ be an action with the Rokhlin property,
let $\ep>0$ and let $F\subseteq A$ be a finite subset. Then there exists a unitary $u\in A$ such that
\be\item $\alpha_\zeta(u)=\zeta u$ for all $\zeta\in\T$.
\item $\|ua-au\|<\varepsilon$ for all $a\in F$.\ee\end{prop}
The definition of the Rokhlin property differs from the conclusion of this proposition in that in condition (1), one only requires
$\|\alpha_\zeta(u)-\zeta u\|<\varepsilon$ for all $\zeta\in\T$.
\begin{proof} One can normalize $F$ so that $\|a\|\leq 1$ for all $a\in F$.
Set $\ep_0=\min\left\{\frac{1}{3},\frac{\ep}{7+2\ep}\right\}$. Choose a unitary $v\in A$ such that conditions (1) and (2) in
\autoref{def RP} are satisfied for the finite set $F$ with
$\varepsilon_0$ in place of $\varepsilon$. Denote by $\mu$ the normalized Haar measure on $\T$, and set
$$x =\int_{\T}\overline{\zeta}\alpha_\zeta(v)\ d\mu(\zeta).$$
Then $\|x\|\leq 1$ and $\|x-v\|\leq \varepsilon_0$. One checks that $\|x^*x-1\|\leq 2\varepsilon_0< 1$ and that $\alpha_\zt(x)=\zt x$
for all $\zt\in \T$. In particular, $x^*x$ is invertible. \\
\indent We have
$$\left\|(x^*x)^{-1}\right\|\leq \frac{1}{1-\|1-x^*x\|}\leq\frac{1}{1-2\ep_0},$$
and thus $\left\|(x^*x)^{-\frac{1}{2}}\right\|\leq\frac{1}{\sqrt{1-2\ep_0}}$.\\
\indent Set $u = x(x^*x)^{-\frac{1}{2}}$, which is a unitary in $A$. Using that $\|x\|\leq 1$ at the first step,
and that $0\leq 1-(x^*x)^{\frac{1}{2}}\leq 1-x^*x$ at the third step, we get
\begin{align*} \|u-x\| & \leq \left\|(x^*x)^{-\frac{1}{2}}-1\right\|\\
& \leq \left\|(x^*x)^{-\frac{1}{2}}\right\|\left\|1-(x^*x)^{\frac{1}{2}}\right\|\\
& \leq \frac{1}{\sqrt{1-2\ep_0}}\|1-x^*x\|\leq \frac{2\ep_0}{\sqrt{1-2\ep_0}}.
\end{align*}
We deduce that
$$ \|u-v\|\leq \frac{2\ep_0}{\sqrt{1-2\ep_0}}+\ep_0.$$
For $\zt\in \T$, we have $\alpha_\zt(x^*x)=x^*x$ and hence $\alpha_\zeta(u)=\zeta u,$
so $u$ satisfies condition (1) of the statement. Finally, for $a\in F$, we have
\begin{align*} \|ua-au\| & \leq \|ua-va\|+\|va-av\|+\|av-au\|\\
&< \|u-v\|\|a\|+\ep_0+\|a\|\|v-u\|\\
&\leq  \frac{4\ep_0}{\sqrt{1-2\ep_0}}+3 \ep_0< \frac{4\ep_0}{1-2\ep_0}+3\ep_0\\
&< \frac{7\ep_0}{1-2\ep_0}<\ep,
\end{align*}
as desired.
\end{proof}

\begin{rem}\label{rem:LtEqSj} In the language of \cite{PhiSorThi_eqsj}, the proof of \autoref{can replace estimate by equality}
shows that the action of $\T$ on $C(\T)$ induced by group multiplication, is equivariantly semiprojective.
This fact seems not to have been known before. \end{rem}

We now turn to examples of circle actions with the Rokhlin property. As in the finite group case, the \Rp\ is rare, and it is challenging to
construct many examples on simple \ca s. We will give an explicit construction of a family of circle actions with the \Rp\ on simple \ATa s,
and also on the Cuntz algebra $\Ot$. For more examples on purely inifinite \ca s, see \cite{Gar_Kir2} (the construction of the
examples there is not explicit).

\begin{eg}\label{eg of Rp on simple AT-algebra} This is an example of a circle action on a simple, unital \ATa\ with the Rokhlin property. For
$n\in\N$, set $A_n=C(\T )\otimes M_{n!}$. Consider the action $\alpha^{(n)}\colon\T\to\Aut(A_n)$ given by $\alpha^{(n)}_\zeta(f)(w)=f(\zeta^{-1}w)$
for $\zeta$ and $w\in \T$ and for $f\in A_n\cong C(\T ,M_{n!})$. In other words, $\alpha^{(n)}$ is the tensor product of the action of left
translation of $\T $, with the trivial action on $M_{n!}$. Then $\alpha^{(n)}$ has the Rokhlin property by \autoref{Properties RP}, since
the action of left translation of $\T$ on itself trivially has the Rokhlin property.\\
\indent We construct a direct limit algebra $A=\dirlim (A_n,\iota_n)$ as follows. Fix a countable dense subset $X=\{x_1,x_2,x_3,\ldots\}\subseteq
\T $, and assume that $x_1=1$. With $f_{x}(\zeta)=f(x^{-1}\zeta)$ for $f\in A_n$, for $x\in X$ and for $\zeta\in \T $, define maps $\iota_n\colon A_n\to
A_{n+1}$ for $n\in\N$, by
$$\iota_n(f)=\left(
               \begin{array}{cccc}
                 f_{1} & 0 & \cdots & 0 \\
                 0 & f_{x_2} & \cdots & 0 \\
                 \vdots & \vdots& \ddots & \vdots \\
                 0 & 0 & \cdots & f_{x_n} \\
               \end{array}
             \right)$$
for every $f\in A_n$.Then $\iota_n$ is unital and injective, for all $n\in\N$. The limit algebra $A=\dirlim(A_n,\iota_n)$ is a unital \ATa.\\
\indent It is easy to check that
$$\iota_n\circ\alpha^{(n)}_\zeta=\alpha^{(n+1)}_\zeta\circ\iota_n$$
for all $n\in\N$ and all $\zeta\in\T$, so
that $\left(\alpha^{(n)}\right)_{n\in\N}$ induces a direct limit action $\alpha=\varinjlim \alpha^{(n)}$ of $\T$ on $A$. Then $\alpha$ has the
Rokhlin property by \autoref{direct limit and Rp}. Simplicity of $A$ follows from Proposition 2.1 in \cite{DadNagNemPas_RedTopSR},
since $X$ is assumed to be dense in $\T$.\end{eg}

In the example above, the universal UHF-pattern can be replaced by any other UHF or (simple)
AF-pattern, and the resulting \ca\ is also a (simple) A$\T$-algebra.

Using the absorption properties of $\Ot$, we can construct an action of the circle on $\Ot$ with the \Rp.

\begin{eg}\label{eg of Rp on Ot} Let $A$ and $\alpha$ be as in the example above. Then $A$ is a separable, unital, nuclear, simple \ca. Use Theorem 3.8 in \cite{KirPhi_embedding} to choose an isomorphism $\varphi\colon A\otimes \Ot\to \Ot$, and define an action $\gamma\colon\T\to\Aut(\Ot)$ by $\gamma_\zeta=\varphi\circ (\alpha_\zeta\otimes \id_{\Ot})\circ \varphi^{-1}$ for $\zeta\in\T$. Since $\alpha$ has the \Rp, it follows from \autoref{Properties RP} that $\gamma$ has the \Rp\ as well.\end{eg}

\begin{eg} If $A$ is any unital \ca\ such that $A\otimes \Ot\cong A$, then one can construct a circle action on $A$ with the Rokhlin property by tensoring the trivial action on $A$ with any action on $\Ot$ with the \Rp, such as the one constructed in \autoref{eg of Rp on Ot}.\end{eg}

\subsection{Nonexistence of actions with the Rokhlin property} Our next goal is to prove that UHF-algebras do not admit any direct limit action of
the circle with the \Rp; see \autoref{no direct limit actions with the Rp on AF-algebras}. We begin with an easy lemma which already rules out
such actions on matrix algebras; see \autoref{cor: no actions on matrices}.

\begin{lma} \label{lma: pointwise outer}
Let $A$ be a \uca\ and let $\alpha\colon \T\to\Aut(A)$ be an action with the \Rp. Then $\alpha_\zt$ is not inner for all $\zt\in \T$
 with $\zt\neq 1$.\end{lma}
\begin{proof}
Let $\zt\in \T\setminus\{1\}$, and assume that there exists a unitary $v\in A$ such that $\alpha_\zt=\Ad(v)$. Let $\ep>0$ satisfy
$\ep<\frac{|1-\zt|}{2}$. Using the \Rp\ for $\alpha$, find a unitary $u\in A$ such that $\|\alpha_\zt(u)-\zt u\|<\ep$
and $\|uv-vu\|<\ep$. Then
\begin{align*}
\ep &>\|\alpha_\zt(u)-\zt(u)\|=\|vuv^*-\zt u\|\geq \left| \|u-\zt u\| - \|vuv^*-u\|\right|\\
&=|1-\zt|-\|vu-uv\|> \frac{|1-\zt|}{2}>\ep,
\end{align*}
which is a contradiction. This shows that $\alpha_\zt$ is not inner.
\end{proof}

\begin{cor} \label{cor: no actions on matrices}
Let $n\in \N$. Then there are no actions of the circle on $M_n$ with the \Rp.\end{cor}
\begin{proof} This is an immediate consequence of \autoref{lma: pointwise outer}, since every automorphism of $M_n$ is inner.\end{proof}

We will generalize the corollary above in \autoref{no direct limit actions with the Rp on AF-algebras} below, where we show that there are
no direct limit actions of the circle with the \Rp\ on UHF-algebras. We need a series of preliminary results.

\begin{nota} Let $n\in\N$. We denote by $\U_n(\C)$ the unitary group of $M_n$.
Identify $\T$ with the center $\mathcal{Z}(\U_n(\C))$ of $\U_n(\C)$ via the map $\zeta\mapsto \mbox{diag}(\zeta,\ldots,\zeta)$, and denote by $P\U_n(\C)$ the quotient group $P\U_n(\C)=\U_n(\C)/\T $.
\end{nota}

\begin{prop}\label{Lemma actions on Mn} Let $n\in\N$ and let $\gamma\colon\T\to\Aut(M_n)$ be a continuous action. Then there exists a
 continuous map $v\colon \T\to\U_n(\C)$ such that $\gamma_\zeta=\Ad(v(\zeta))$ for all $\zeta\in \T$.
\end{prop}

\begin{proof} Recall that every automorphism of $M_n$ is inner, so that for every $\zeta\in\T$ there exists a unitary $u(\zeta)\in\U_n(\C)$ such that $\alpha_\zeta=\Ad(u(\zeta))$. Moreover, $u(\zeta)$ is uniquely determined up to multiplication by elements of $\T =\mathcal{Z}(\U_n(\C))$ and hence $\gamma_\zeta$ determines a continuous group homomorphism $u\colon \T\to P\U_n(\C)$. Denote by $\rho\colon \U_n(\C)\to P\U_n(\C)$ the canonical projection. We want to solve the following lifting problem:
\beqa\xymatrix{ & \U_n(\C)\ar[d]^-{\rho}\\
\T\ar@{-->}[ur]^-{v}\ar[r]_-{u}&P\U_n(\C).}\eeqa
The map $u$ determines an element $[u]\in \pi_1(P\U_n(\C))$ and $\rho$ induces a group homomorphism $\pi_1(\rho)\colon \pi_1(\U_n(\C))\to\pi_1(P\U_n(\C))$. The quotient map $\rho\colon \U_n(\C)\to P\U_n(\C)$ is actually a fiber bundle, since $\U_n(\C)$ is a manifold and the action of $\T $ on $\U_n(\C)$ is free. See the theorem in Section 4.1 of \cite{Pal_slicesNonCptLieGps}. The long exact sequence in homotopy for this fiber bundle is
\beqa\xymatrix{ \cdots \ar[r] &\pi_1(\T )\ar[r] & \pi_1(\U_n(\C))\ar[r]^-{\pi_1(\rho)} &\pi_1(P\U_n(\C)) \ar[r]& \pi_0(\T ).}\eeqa
Recall that $\pi_1(\U_n(\C))\cong\Z$, and that $\pi_0(\T )\cong 0$. The map $\pi_1(\T )\to \pi_1(\U_n(\C))$ is induced by $\zeta\mapsto \diag(\zeta,\ldots,\zeta)$, which on $\pi_1$ corresponds to multiplication by $n$. In other words, the above exact sequence is
\beqa\xymatrix{ \cdots \ar[r] &\Z\ar[r]^-{\cdot n} & \Z\ar[r]^-{\pi_1(\rho)} &\pi_1(P\U_n(\C)) \ar[r]& 0,}\eeqa
which implies that $\pi_1(P\U_n(\C))\cong \Z_n$ and that the map $\pi_1(\rho)$ is surjective. It follows that $u$ is homotopic to a map $\widehat{u}\colon \T\to P\U_n(\C)$ that is liftable. The homotopy lifting property for fiber bundles implies that $u$ itself is liftable, that is, there exists a continuous map $v\colon \T\to\U_n(\C)$ such that $u(\zeta)=\rho(v(\zeta))$ for all $\zeta\in\T$. (See the paragraph below Theorem 4.41 in \cite{Hatcher} for the definition of the homotopy lifting property. Proposition 4.48 in \cite{Hatcher} shows that every fiber bundle has this property.) This concludes the proof. \end{proof}

\begin{lma}\label{unitaries in Mn at least 2 away} Let $n\in\N$ and let $v\colon \T\to \U_n(\C)$ be a continuous map.
Then for every $u\in \U_n(\C)$, there exists $\zeta\in \T$ such that
$$\|v(\zeta)uv(\zeta)^*-\zeta u\|\geq 2.$$\end{lma}
\begin{proof} Assume that there exists $u\in\U_n(\C)$ such that $\|v(\zeta)uv(\zeta)^*-\zeta u\|< 2$ for all $\zeta\in\T$. Define $w\in C(\T,M_n)$ by
$w(\zeta)=\overline{\zeta}v(\zeta)uv(\zeta)^*u^*$ for all $\zt\in \T$. Then $w$ is a unitary in $C(\T,M_n)$ and
$\left\|w-1_{C(\T,M_n)}\right\|<2$. It follows that the spectrum of $w$ is not the whole circle, and a standard functional calculus argument using a branch of the logarithm shows that
$w$ is an exponential, that is, there is a self-adjoint element $x\in C(\T,M_n)$ with $w=e^{ix}$ (namely $x=\log(w)$). 
Then the path $t\mapsto e^{itx}$, for $t\in [0,1]$, defines a homotopy between $w$ and $1_{C(\T,M_n)}$. Define now a continuous function
$f\colon\T\to\T$ by $f=\det \circ w$. Then $f$ is homotopic to the constant map, and thus its winding number is zero.\\
\indent On the other hand,
$$f(\zeta)=\det(w(\zeta))=\det(\overline{\zeta}v(\zeta)uv(\zeta)^*u^*)=\overline{\zeta}^n,$$
so the winding number is actually $-n$. This is a contradiction, and the result follows.\end{proof}

\begin{thm}\label{no direct limit actions with the Rp on AF-algebras} Assume that $A=\varinjlim(M_{k_n},\iota_n)$ is an unital UHF-algebra with
unital connecting maps. If $\alpha=\varinjlim \alpha^{(n)}$ is a direct limit action of the circle on $A$, then $\alpha$ does not have the
Rokhlin property.\end{thm}
\begin{proof} Assume that $\alpha$ has the Rokhlin property. Take $\ep =2$ and $F=\emptyset$. A standard approximation argument
shows that there exist $n\in\N$ and $u\in\U_{k_n}(\C)$ such that
$$\left\|\alpha^{(n)}_\zeta(u)-\zeta u\right\|<2$$
for all $\zeta\in \T$. By \autoref{Lemma actions on Mn}, there is a continuous map $v\colon \T\to\U_{k_n}(\C)$ such that
$\alpha^{(n)}_\zeta=\Ad(v(\zeta))$ for all $\zeta\in\T$. Now, \autoref{unitaries in Mn at least 2 away} implies that there exist $\zeta_0\in \T$ such that
$\|v(\zeta_0)uv(\zeta_0)^*-\zeta_0 u\|\geq 2$. Therefore, $2> \left\|\alpha^{(n)}_{\zeta_0}(u)-\zeta_0 u\right\|\geq 2$, which is a contradiction.
Thus, $\alpha$ does not have the Rokhlin property. \end{proof}

\subsection{Genericity on $\mathcal{O}_2$-absorbing algebras}

In this subsection, we specialize to circle actions with the \Rp\ on $\Ot$-absorbing \ca s. This class of \ca s is special in our context. Indeed, circle actions with the \Rp\ are generic on separable, unital, $\Ot$-absorbing \ca s; see \autoref{rokhlin action are a Gdelta set}.
This fact should be contrasted with \autoref{no direct limit actions with the Rp on AF-algebras}.
Additionally, since $\mathcal{O}_2$ absorbs every UHF-algebra, \autoref{characterize restriction has Rp} applies to
algebras that absorb $\Ot$, so the restriction of every circle action with the \Rp\ to a finite subgroup,
again has the \Rp.

This subsection is devoted to proving the first of the two results mentioned above.
Throughout, $A$ will be a separable, unital \ca.

\begin{df} Given an enumeration $S=\{a_1,a_2,\ldots\}$ of a countable dense subset of the unit ball of $A$, define metrics on $\Aut(A)$ by
$$\rho_S^{(0)}(\alpha,\beta)=\sum\limits_{k=1}^\I \frac{\|\alpha(a_k)-\beta(a_k)\|}{2^k} \ \ \ \mbox{ and } \ \ \ \rho_S(\alpha,\beta)=\rho_S^{(0)}(\alpha,\beta)+\rho_S^{(0)}(\alpha^{-1},\beta^{-1}).$$\end{df}

Denote by $\mbox{Act}_\T(A)$ the set of all circle actions on $A$. For any enumeration $S=\{a_1,a_2,\ldots\}$ as above, define a metric on $\mbox{Act}_\T(A)$ by
$$\rho_{\T,S}(\alpha,\beta)=\max_{\zeta\in\T} \rho_S(\alpha_\zeta,\beta_\zeta).$$

\begin{lma}\label{space of T-actions is complete} For any $S$ as above, the function $\rho_{\T,S}$ is a complete metric on
$\mbox{Act}_\T(A)$.\end{lma}
\begin{proof} Let $\left(\alpha^{(n)}\right)_{n\in \N}$ be a Cauchy sequence in $\mbox{Act}_\T(A)$, that is, for every $\varepsilon>0$
there is $n_0\in \N$ such that for every $n,m\geq n_0$, we have $\rho_{\T,S}\left(\alpha^{(n)},\alpha^{(m)}\right)<\varepsilon$. We
want to show that there is $\alpha\in\mbox{Act}_\T(A)$ such that $\lim\limits_{n\to\I}\rho_{\T,S}\left(\alpha,\alpha^{(n)}\right)= 0$.\\
\indent Given $\zeta\in\T$, we have $\rho_S\left(\alpha_\zeta^{(n)},\alpha_\zeta^{(m)}\right)\leq \rho_{\T,S}\left(\alpha^{(n)},\alpha^{(m)}\right)$,
and hence $\left(\alpha_\zeta^{(n)}\right)_{n\in \N}$ is Cauchy in $\Aut(A)$. By Lemma 3.2 in \cite{Phi_generic},
the pointwise norm limit of the sequence $\left(\alpha^{(n)}_\zeta\right)_{n\in \N}$ exists, and we denote it by $\alpha_\zeta$. It
also follows from Lemma 3.2 in \cite{Phi_generic} that $\alpha_\zeta$ is an automorphism of $A$, with inverse
$\alpha_{\zeta^{-1}}$. Moreover, the map $\alpha\colon\T\to\Aut(A)$ given by $\zeta\mapsto \alpha_\zeta$ is a group homomorphism, since
it is the pointwise norm limit of group homomorphisms. It remains to check that it is continuous, and this follows from an $\frac{\ep}{3}$
argument from $\lim\limits_{n\to\I}\left\|\alpha_\zeta^{(n)}(a_k)-\alpha_\zeta(a_k)\right\|= 0$ for all $k\in \N$, and the fact that
$\alpha^{(n)}\colon\T\to\Aut(A)$ is continuous for all $n\in\N$.\end{proof}

\begin{nota} Given a finite subset $F\subseteq A$ and $\varepsilon>0$, let $W_\T(F,\varepsilon)$ be the set of all actions $\alpha\in\mbox{Act}_\T(A)$ such that there exists $u\in\U(A)$ with $\|ua-au\|<\varepsilon$ for all $a\in F$ and $\|\alpha_\zeta(u)-\zeta u\|<\varepsilon$ for all $\zeta\in\T$.\end{nota}

It is easy to check that an action $\alpha\in \mbox{Act}_\T(A)$ has the \Rp\ \ifo $\alpha\in W_\T(F,\varepsilon)$ for all finite subsets
$F\subseteq A$ and all positive numbers $\varepsilon>0$.

\begin{lma}\label{lma} Let $S$ be a countable dense subset of the unit ball of $A$, and let $\mathcal{F}$ be the set of all finite subsets of $S$. Then $\alpha\in\mbox{Act}_\T(A)$ has the \Rp\ \ifo
$$\alpha\in \bigcap_{F\in \mathcal{F}}\bigcap_{n=1}^\I W_\T\left(F,\frac{1}{n}\right).$$\end{lma}
\begin{proof} One just needs to approximate any finite set by scalar multiples of elements in a finite subset of $S$. We omit the details.\end{proof}

Using the notation of the lemma above, observe that the family $\mathcal{F}$ is countable.

\begin{prop}\label{W(F,epsilon) is dense} Let $A$ and $\D$ be unital, separable \ca s, such that there is an action
$\gamma\colon\T\to\Aut(\D)$ with the \Rp. Suppose that there exists an isomorphism $\varphi\colon A\otimes \D\to A$
such that $a\mapsto \varphi(a\otimes 1_\D)$ is approximately unitarily equivalent to $\id_A$. Then for every finite subset
$F\subseteq A$ and every $\varepsilon>0$, the set $W_\T(F,\varepsilon)$ is open and dense.\end{prop}
\begin{proof} We first check that $W_\T(F,\varepsilon)$ is open. Fix an enumeration $S=\{a_1,a_2,\ldots\}$ of a countable
dense subset of the unit ball of $A$. Let $\alpha\in W_\T(F,\varepsilon)$, and choose $u\in\U(A)$ such that $\|ua-au\|<\varepsilon$ for all
$a\in F$ and $\|\alpha_\zeta(u)-\zeta u\|<\varepsilon$ for all $\zeta\in\T$. Set
$$\varepsilon_0=\max_{\zeta\in\T}\|\alpha_\zeta(u)-\zeta u\|,$$
so that $\varepsilon_1=\varepsilon-\varepsilon_0>0$. Choose $k\in \N$ such that $\|a_k-u\|<\frac{\varepsilon_1}{3}$. Now, we claim that if
$\alpha'\in \mbox{Act}_\T(A)$ satisfies $\rho_{\T,S}(\alpha',\alpha)<\frac{\varepsilon_1}{2^{k}3}$, then $\alpha'\in W_\T(F,\varepsilon)$. Indeed,
\begin{align*}\|\alpha'_\zeta(u)-\zeta u\|&\leq \|\alpha'_\zeta(u)-\alpha_\zeta(u)\|+\|\alpha'_\zeta(u)-\zeta u\|\\
&\leq \frac{2\ep_1}{3} +\|\alpha_\zeta'(a_k)-\alpha_\zeta(a_k)\|+ \varepsilon_0\\
&\leq \frac{2\ep_1}{3} + 2^k\rho_{\T,S}(\alpha,\alpha')+\varepsilon_0\\
&=\varepsilon_1+\varepsilon_0=\varepsilon.\end{align*}
This proves that $W_\T(F,\varepsilon)$ is open.\\
\indent We will now show that $W_\T(F,\varepsilon)$ is dense in $\mbox{Act}_\T(A)$. Let $\alpha$ be an arbitrary action in $\mbox{Act}_\T(A)$,
let $T\subseteq A$ be a finite set, and let $\delta>0$. We want to find $\beta\in\mbox{Act}_\T(A)$ such that $\beta\in W_\T(F,\varepsilon)$
and $\rho_{\T,S}(\alpha,\beta)<\delta$. \\
\indent Choose $\delta'>0$ such that $\delta'<\min\{\delta,\varepsilon\}$. Since $\alpha$ is continuous, there is $\delta_0>0$ such that
whenever $\zeta,\zeta'\in \T$ and $|\zeta-\zeta'|<\delta_0$, then $\|\alpha_{\zeta}(a)-\alpha_{\zeta'}(a)\|<\frac{\dt'}{4}$ for all $a\in T$.
Choose $m\in\N$ and $\zeta_1,\ldots,\zeta_m\in\T $ such that for every $\zeta\in\T$ there is $j\in\N$ with $1\leq j\leq m$ and such that
$|\zeta-\zeta_j|<\delta_0$. Choose $w\in\U(A)$ such that $\|w\varphi(1\otimes a)w^*-a\|<\frac{\dt'}{2}$ for all
$a\in T\cup \bigcup\limits_{j=1}^m\alpha_{\zeta_j}(T)$. Set $\psi=\Ad(w)\circ\varphi$ and for $\zeta\in\T$, define an action
$\beta\in\mbox{Act}_\T(A)$ by
$$\beta_\zeta=\psi\circ(\gamma_\zeta\otimes\alpha_\zeta)\circ\psi^{-1}.$$
We claim that $\beta\in W_\T(F,\varepsilon)$. Choose $w'\in A\otimes\D$ of the form $w'=\sum\limits_{\ell=1}^rx_\ell\otimes d_\ell$ for
some $d_1,\ldots,d_r\in \D$ and some $x_1,\ldots,x_r\in A$, such that $\|w-w'\|<\frac{\dt}{3}$. Since $\gamma$ has the \Rp, use
\autoref{can replace estimate by equality} to choose $u\in\U(\D)$ such that $\gamma_\zeta(u)=\zeta u$ for all $\zeta\in\T$ and $\|ud_\ell-d_\ell u\|<\frac{\ep}{4}$ for all $\ell=1,\ldots,r$. Then
$$\|(1_A\otimes u)w'-w'(1_A\otimes u)\|<\frac{\dt}{3}$$
and hence $\|(1_A\otimes u)w-w(1_A\otimes u)\|<\delta$. Set $v=\varphi(1_A\otimes u)$. Then
\begin{align*}\|\beta_\zeta(v)-\zt v\| &=\|w\varphi\left((\alpha_\zeta\otimes\gamma_\zeta)(\varphi^{-1}(w^*\varphi(1_A\otimes u)w))\right)w^*-\zt\varphi(1_A\otimes u)\|\\
& \leq \|w\varphi\left((\alpha_\zeta\otimes\gamma_\zeta)(\varphi^{-1}(w^*\varphi(1_A\otimes u)w))\right)w^*-w\varphi\left((\alpha_\zeta\otimes\gamma_\zeta)(1_A\otimes u)\right)w^*\|\\
& \ \ \ +\|w\varphi\left((\alpha_\zeta\otimes\gamma_\zeta)(1_A\otimes u)\right)w^*-\zt\varphi(1_A\otimes u)\|\\
&<\frac{\dt'}{2} + \|w\varphi\left(\zeta 1_A\otimes u\right)w^*-\zt\varphi(1_A\otimes u)\|\\
&<\frac{\dt'}{2}+\frac{\dt'}{2}=\delta'<\varepsilon\end{align*}
for all $\zeta\in\T$, and thus $\|\beta_\zeta(u)-\zeta v\|<\varepsilon$ for all $\zeta\in\T$. On the other hand, given $a\in F$, we have
\begin{align*}\|va-av\|& = \|\varphi(1_A\otimes u)a-a\varphi(1_A\otimes u)\| \\
&\leq \|\varphi(1_A\otimes u)a-\varphi(1_A\otimes u)w\varphi(a\otimes 1_\D)w^*\|\\
& \ \ \ +\|\varphi(1_A\otimes u)w\varphi(a\otimes 1_\D)w^*-w\varphi(a\otimes 1_\D)w^*\varphi(1_A\otimes u)\|\\
& \ \ \ +\|w\varphi(a\otimes 1_\D)w^*\varphi(1_A\otimes u)-a\varphi(1_A\otimes u)\|\\
&<\frac{\dt'}{2}+0+\frac{\dt'}{2}=\delta'<\varepsilon\end{align*}
because $a\otimes 1_\D$ and $1_A\otimes u$ commute. This proves the claim.\\
\indent It remains to prove that $\|\beta_\zeta(a)-\alpha_\zeta(a)\|<\delta$ for all $a\in T$ and all $\zt\in \T$. For fixed $\zeta\in\T$ and $a\in T$, we have
\begin{align*}\|\beta_\zeta(a)-\alpha_\zeta(a)\|&=\|w\varphi\left((\alpha_\zeta\otimes\gamma_\zeta)(\varphi^{-1}(w^*aw))\right)-\alpha_\zeta(a)\| \\
&\leq \|w\varphi\left((\alpha_\zeta\otimes\gamma_\zeta)(\varphi^{-1}(w^*aw))\right)-w\varphi\left((\alpha_\zeta\otimes\gamma_\zeta)(a\otimes 1_\D)\right)\|\\
& \ \ \ + \|w\varphi\left((\alpha_\zeta\otimes\gamma_\zeta)(a\otimes 1_\D)\right)-\alpha_\zeta(a)\|\\
&< \frac{\dt'}{2}+\|w\varphi(\alpha_\zeta(a)\otimes 1_\D)w^*-\alpha_\zeta(a)\|\\
&\leq \frac{\dt'}{2}+ \|w\varphi(\alpha_\zeta(a)\otimes 1_\D)w^*-w\varphi(\alpha_{\zeta_j}(a)\otimes 1_\D)w^*\|\\
&\ \ \ +\|w\varphi(\alpha_{\zeta_j}(a)\otimes 1_\D)w^*-\alpha_{\zeta_j}(a)\|+\|\alpha_{\zeta_j}(a)-\alpha_{\zeta}(a)\|\\
&<\frac{\dt'}{2}+\frac{\dt'}{4}+\frac{\dt'}{4}=\delta'<\delta.\end{align*}
This finishes the proof.\end{proof}

\begin{thm}\label{Rokhlin are generic on D-absorbing algs} Let $A$ and $\D$ be unital, separable \ca s, such that there is an action
$\gamma\colon\T\to\Aut(\D)$ with the \Rp. Suppose that there exists an isomorphism $\varphi\colon A\otimes \D\to A$ such that
$a\mapsto \varphi(a\otimes 1_\D)$ is approximately unitarily equivalent to $\id_A$. Then the set of all circle actions with the \Rp\ on $A$ is a dense
$G_\delta$-set in $\mbox{Act}_\T(A)$. \end{thm}
\begin{proof} By \autoref{lma}, the set of all circle actions on $A$ that have the \Rp\ is precisely the countable intersection
$$\bigcap_{F\in \mathcal{F}}\bigcap_{n\in\N} W_\T\left(F.\frac{1}{n}\right).$$
By \autoref{W(F,epsilon) is dense}, each $W_\T\left(F,\frac{1}{n}\right)$ is open and dense in $\mbox{Act}_\T(A)$, which is a complete metric space
by \autoref{space of T-actions is complete}, so the result follows from the Baire Category Theorem.\end{proof}

Recall that a unital, separable \ca\ $\D$ is said to be strongly self-absorbing if it is infinite-dimensional and the map $\D\to \D\otimes \D$ given by $d\mapsto d\otimes 1$ is approximately unitarily equivalent to an isomorphism. (Strongly self-absorbing \ca s are always nuclear, so there is no ambiguity when talking about tensor products.) The only known examples are the Jiang-Su algebra $\mathcal{Z}$, the Cuntz algebras $\Ot$ and $\mathcal{O}_\I$, UHF-algebras of infinite type, and tensor products of $\OI$ by such UHF-algebras. See \cite{TomWin_SSA} for more details and results on strongly self-absorbing \ca s.

\begin{rem}\label{remark} In the context of \autoref{Rokhlin are generic on D-absorbing algs}, suppose additionally that $\D$ is a unital,
separable strongly self-absorbing \ca. Then, according to Theorem 7.2.2 in \cite{Ror_BookClassif}, \tfae:
\be\item There exists an isomorphism $\varphi\colon A\otimes\D\to A$ such that $a\mapsto \varphi(a\otimes 1_\D)$ is approximately unitarily equivalent to $\id_A$;
\item There exists some isomorphism $\psi\colon A\otimes \D\to A$.\ee\end{rem}

\begin{thm}\label{rokhlin action are a Gdelta set} Let $A$ be a separable \uca\ such that $A\otimes\Ot\cong A$. Then the set of all circle actions on $A$ with the \Rp\ is a dense $G_\delta$-set in $\mbox{Act}_\T(A)$.\end{thm}
\begin{proof} By \autoref{eg of Rp on Ot}, there is an action $\gamma\colon\T\to\Aut(\Ot)$ with the \Rp. Since $A$ absorbs $\Ot$ tensorially,
the hypotheses of \autoref{Rokhlin are generic on D-absorbing algs} are met by \autoref{remark}, and the result follows. \end{proof}

It is a consequence of the theorem above that the \Rp\ is generic for circle actions on $\Ot$. Nevertheless, we do not know of any such action for which it is possible to describe what the images of the canonical generators of $\Ot$ are. In particular, we do not have a model action on $\Ot$.

\section{Restrictions to finite cyclic groups}
This section is devoted to proving that for $n\in\N$,
the restriction of a circle action with the \Rp\ on a $M_{n^\I}$-absorbing \ca\ to the finite cyclic
group $\Z_n$ again has the \Rp. See \autoref{characterize restriction has Rp}. This phenomenon cannot be expected to hold in full generality since the
\Rp\ for a circle action does not guarantee the existence of any non-trivial projections. Even more, there are serious $K$-theoretical
obstructions to the \Rp\ for finite groups. See \autoref{eg restriction on S1} and \autoref{eg restriction on pi ca} below.\\
\indent On the other hand, this result will be used in subsequent work to classify
circle actions with the Rokhlin property on unital \ca s that absorb some UHF-algebra of infinite type.

We give a rough outline of what our strategy will be. We will first focus on cyclic group actions which are restrictions of circle
actions with the \Rp. These have what we call the ``unitary Rokhlin property'', which is a weakening of the \Rp\ of
\autoref{RP}, that asks
for a unitary instead
of projections; see \autoref{df: uRp}. Dual actions of actions with the unitary Rokhlin property can be completely characterized, and we
do so in \autoref{strongly approx inner and circle Rp}. The relevant notion is that of ``strong approximate innerness''; see
\autoref{df: sai}.
We will later show in (the proof of) \autoref{characterize restriction has Rp} that, under a number of assumptions, every strongly approximately inner
action of $\Z_n$ is approximately representable, which is the notion dual to the \Rp, as was shown by Izumi in \cite{Izu_RpI}.
The conclusion is then that the original restriction, which a priori had the unitary Rokhlin property, actually has the \Rp.

The following is Definition 3.6 in \cite{Izu_RpI}.

\begin{df}\label{df: sai} Let $B$ be a \uca, and let $\beta$ be an action of a finite abelian group $G$ on $B$.
\be\item We say that $\beta$ is \emph{strongly approximately inner} if there exist unitaries $u(g)\in (B^\beta)^\I$, for  $g\in G$, such that
$$\beta_g(b)=u(g)bu(g)^*$$
for $b\in B$ and $g\in G$.
\item We say that $\beta$ is \emph{approximately representable} if $\beta$ is strongly approximately inner and the unitaries
$u(g)$ for $g\in G$ as in (1) above, can be chosen to form a representation of $G\in (B^\beta)^\I$.\ee\end{df}

\begin{nota}\label{generating automorphism} Let $B$ be a \ca, let $G$ be a cyclic group (that is, either $\Z$ or $\Z_n$ for some $n\in \N$), and let 
$\beta\colon G\to\Aut(B)$ be action of $G$ on $B$. We will usually make a slight abuse of notation and also denote by $\beta$ the generating automorphism $\beta_1$.\end{nota}

If $G$ is a finite cyclic group, we have the following characterization of strong approximate innerness in terms of elements in $B$, rather than in $(B^\beta)^\I$.

\begin{lma} Let $B$ be a separable, \uca, let $n\in\N$, and let $\beta$ be an action of $\Z_n$ on $B$. Then $\beta$ is strongly approximately inner \ifo for every finite subset $F\subseteq B$ and every $\varepsilon>0$, there is a unitary $w\in \U(B)$ such that $\|\beta(w)-w\|<\varepsilon$ and $\|\beta(b)-wbw^*\|<\varepsilon$ for all $b\in F$. Moreover, $\beta$ is approximately representable \ifo the unitary $w$ above can be chosen so that $w^n=1$.\end{lma}
\begin{proof} Assume that $\beta$ is strongly approximately inner. Use a standard perturbation argument to choose a sequence
$(u_m)_{m\in\N}$ of unitaries in $B^\beta$ that represents $u(1)\in (B^\beta)^\I$. Then $\lim\limits_{m\to\I}\|\beta(u_m)-u_m\|= 0$,
and for $b\in B$, we have $\lim\limits_{m\to\I}\|\beta(b)-u_mbu_m^*\|= 0$. \\
\indent Given a finite set $F\subseteq B$ and $\varepsilon>0$, choose
$M\in\N$ such that $\|\beta(u_M)-u_M\|<\varepsilon$ and $\|\beta(b)-u_Mbu_M^*\|<\varepsilon$ for all $b\in F$, and set $w=u_M$. \\
\indent Conversely, let $(F_m)_{m\in\N}$ be an increasing sequence of finite subsets of $B$ satisfying $\overline{\bigcup\limits_{m\in\N}F_m}=B$. 
For $m\in \N$, set $\varepsilon_m=\frac{1}{m}$ and let $w_m$ be as in the statement for $\ep_m$ and $F_m$.
Then
$$u=\overline{(w_m)}_{m\in\N} \in (B^\beta)^\I$$
satisfies $\beta(b)=ubu^*$ for all $b\in F_m$, and hence $\beta$ is strongly approximately inner.\\
\indent For the second statement, observe that a unitary of order $n\in (B^\beta)^\I$ can be lifted to a sequence unitaries of order
$n$ in $B^\beta$. First, observe that a unitary in $(B^\beta)^\I$ can always be lifted to a sequence of unitaries in $B$. Second, a standard functional 
calculus argument shows that if $v$ is a unitary in $B^\beta$ such that $\|v^n-1\|$ is small,
then $v$ is close to a unitary $\widetilde{v}\in B^\beta$ such that $\widetilde{v}^n=1$. We omit the details. \end{proof}

The following is Lemma 3.8 in \cite{Izu_RpI}.

\begin{lma} \label{approx repres and Rp} Let $B$ be a separable \uca, and let $\beta$ be an action of a finite abelian group $G$ on $B$.
\be\item The action $\beta$ has the \Rp\ \ifo the dual action $\widehat{\beta}$ is approximately representable.
\item The action $\beta$ is approximately representable \ifo the dual action $\widehat{\beta}$ has the \Rp.\ee\end{lma}

The lemma above should be regarded as the assertion that for finite abelian group actions, the \Rp\ and approximate representability are dual notions.
It is therefore natural to ask what condition on $\beta$ is equivalent to its dual action being strongly approximately inner, rather than
approximately representable. Such a condition will necessarily be weaker than the \Rp. We define the relevant property below.

\begin{df}\label{df: uRp}
Let $B$ be a \uca, let $n\in\N$ and let $\beta\colon\Z_n\to\Aut(B)$ be an action. We say that $\beta$ has the \emph{unitary \Rp} if for every
$\varepsilon>0$ and for every finite subset $F\subseteq B$, there exists $u\in\U(B)$ such that $\|ub-bu\|<\varepsilon$ for all $b\in F$ and
$\left\|\beta_k(u)-e^{2\pi i k/n}u\right\|<\varepsilon$ for all $k\in\Z_n$.\end{df}

Let $A$ be a \uca. Given a continuous action $\alpha\colon\T\to\Aut(A)$, and $n\in\N$, we denote by $\alpha|_n$ the restriction $\alpha|_{\Z_n}\colon\Z_n\to\Aut(A)$ of $\alpha$ to
$$\{1,e^{2\pi i/n},\ldots,e^{2\pi i (n-1)/n}\}\cong \Z_n.$$

Recall that if $v$ is the canonical unitary in $A\rtimes_{\alpha|_n}\Z_n$ implementing $\alpha|_n$, then the dual action
$$\widehat{\alpha|_n}\colon\Z_n\cong\widehat{\Z_n}\to \Aut(A\rtimes_{\alpha|_n}\Z_n)$$
of $\alpha|_n$ is given by $\left(\widehat{\alpha|_n}\right)_k(a)=a$ for all $a\in A$ and $\left(\widehat{\alpha|_n}\right)_k(v)=e^{2\pi i k/n}v$ for all $k\in\Z_n$.

The following easy lemmas provide us with many examples of cyclic group actions with the unitary \Rp.

\begin{lma}\label{restriction has uRp} If $\alpha\colon\T\to\Aut(A)$ has the \Rp, then $\alpha|_n$ has the unitary \Rp\ for all $n\in\N$.\end{lma}
\begin{proof} Given $\ep >0$ and a finite subset $F\subseteq A$, choose a unitary $u\in \U(A)$ such that $\|ua-au\|<\ep $ for all
$a\in F$ and $\|\alpha_\zeta(u)-\zeta u\|<\varepsilon$ for all $\zeta\in\T$. If $n\in\N$, then
$$\left\|(\alpha|_n)_k(u)-e^{2\pi i k/n}u\right\|=\left\|\alpha_{e^{2\pi i k/n}}(u)-e^{2\pi i k/n}u\right\|<\varepsilon$$
for all $k\in\Z_n$, as desired.\end{proof}

\begin{lma}\label{Rp implies uRp} If $\beta\colon\Z_n\to\Aut(B)$ has the \Rp, then it has the unitary \Rp.\end{lma}
\begin{proof} Given $\varepsilon>0$ and a finite subset $F\subseteq B$, choose projections $e_0,\ldots,e_{n-1}$ as in the definition
of the Rokhlin property for the tolerance
$\frac{\varepsilon}{n}$ and the finite set $F$, and set $u=\sum\limits_{j=0}^{n-1}e^{-2\pi i j/n}e_j$. Then $u$ is a unitary in $B$. Moreover, $\|ub-bu\|<\varepsilon$ for all $b\in F$ and
$$\left\|\beta_k(u)-e^{2\pi i k/n}u\right\|=\left\|\sum\limits_{j=0}^{n-1}e^{2\pi ij/n}\beta_k(e_j)-e^{2\pi i k/n}\sum\limits_{j=0}^{n-1}e^{-2\pi i j/n}e_j\right\|<\varepsilon$$
since $\|\beta_k(e_j)-e_{j+k}\|<\frac{\varepsilon}{n}$ for all $j,k\in\Z_n$, and the projections $e_0,\ldots,e_{n-1}$ are pairwise orthogonal.\end{proof}

The converse of the preceding lemma is not in general true, since the unitary Rokhlin property does not ensure the existence of any non-trivial projections on the algebra. We present two examples of how this can fail.

\begin{eg}\label{eg restriction on S1} Consider the action of left translation of $\T$ on $C(\T )$. It has the Rokhlin property, so its restriction
to any $\Z_n\subseteq \T$ has the unitary \Rp. However, no non-trivial finite group action on $C(\T )$ can have the \Rp\ since $C(\T )$ has no non-trivial
projections. \end{eg}

Besides merely the lack of projections, there are less obvious $K$-theoretic obstructions for the restrictions of an action of the circle with the \Rp\ to
have the \Rp. See \autoref{eg restriction on pi ca}.\\
\indent We need a lemma first.

\begin{prop}\label{trivial action on K-thy} Let $G$ be a connected metric group, let $A$ be a \uca, and let $\alpha\colon G\to\Aut(A)$ be a
continuous action (not necessarily with the Rokhlin property). Then $K_\ast(\alpha_g)=\id_{K_\ast(A)}$ for all $g\in G$.\end{prop}
\begin{proof} We just prove it for $K_0$; the proof for $K_1$ is similar, or follows by replacing $(A,\alpha)$ with $(A\otimes B,\alpha\otimes\id_B)$,
where $B$ is any \ca\ satisfying the UCT such that $K_0(B)=0$ and $K_1(B)=\Z$, and using the K\"unneth formula. (For example, $B=C_0(\R)$ will do.)\\
\indent Denote the metric on $G$ by $d$. Let $n\in\N$ and let $p$ be a projection in $M_n(A)$. Set $\alpha^{(n)}=\alpha\otimes\id_{M_n}$, the
augmentation of $\alpha$ to $M_n(A)$. Since $\alpha^{(n)}$ is continuous, there exists $\delta>0$ such that
$\left\|\alpha^{(n)}_g(p)-\alpha^{(n)}_h(p)\right\|<1$ whenever $g$ and $h\in G$ satisfy $d(g,h)<\delta$. Since $\alpha^{(n)}_g(p)$ and $\alpha^{(n)}_h(p)$
are projections in $M_n(A)$, it follows that $\alpha^{(n)}_g(p)$ and $\alpha^{(n)}_h(p)$ are homotopic, and hence their classes in $K_0(A)$ agree,
that is, $K_0(\alpha_g)([p]_0)=K_0(\alpha_h)([p]_0)$. Denote by $e$ the unit of $G$. Since $g$ and $h$ satisfying $d(g,h)<\delta$ are arbitrary,
and since $G$ is connected, it follows that
$$K_0(\alpha_g)([p]_0)=K_0(\alpha_e)([p]_0)=[p]_0$$
for any $g\in G$. Since $p$ is an arbitrary projection in $A\otimes\K$, it follows that $K_0(\alpha_g)=\id_{K_0(A)}$ for all $g\in G$,
as desired.\end{proof}

\begin{eg}\label{eg restriction on pi ca} This is an example of a purely infinite simple separable nuclear \uca\ (in particular, with many projections), and an action of the circle on it satisfying the \Rp, such that no restriction to a finite subgroup of $\T$ has the \Rp. \\
\indent Let $\{p_n\}_{n\in\N}$ be an enumeration of the prime numbers, and for every $n\in \N$, set $q_n=p_1\cdots p_n$. Fix a countable dense subset $X=\{x_1,x_2,x_3,\ldots\}$ of $\T$ with $x_1=1$. For $x\in X$ and $f\in C(\T)$, denote by $f_x$ the function in $C(\T)$ given by $f_x(\zeta)=f(x^{-1}\zeta)$ for $\zeta\in\T$. For $n\in \N$, define a unital injective map
$$\iota_n\colon M_{q_n}(C(\T))\to M_{q_{n+1}}(C(\T))$$
by $\iota_n(f)=\diag\left(f_{1},f_{x_2},\ldots, f_{x_{p_n}}\right)$
for $f\in M_{q_n}(C(\T))$. The direct limit $A=\varinjlim (M_{q_n}(C(\T)),\iota_n)$ is a unital A$\T$-algebra, and an argument similar to the one exhibited in \autoref{eg of Rp on simple AT-algebra} shows that $A$ is simple. For $n\in\N$, let $\alpha^{(n)}\colon \T\to\Aut(M_{q_n}(C(\T)))$ be the tensor product of the trivial action on $M_{q_n}$ with the action coming from left translation on $C(\T)$. Then $\alpha^{(n)}$ has the \Rp\ by \autoref{Properties RP}. Since $\iota_n\circ\alpha^{(n)}_\zeta=\alpha^{(n+1)}_\zeta\circ\iota_n$ for all $n\in\N$ and all $\zeta\in\T$, the sequence $\left(\alpha^{(n)}\right)_{n\in\N}$ induces a direct limit action $\alpha=\varinjlim\alpha^{(n)}$ of $\T$ on $A$, which has the \Rp\ by \autoref{direct limit and Rp}.\\
\indent Now set $B=A\otimes\OI$ and define $\beta\colon\T\to\Aut(B)$ by $\beta=\alpha\otimes\id_{\OI}$. Then $B$ is a purely infinite, simple, separable, nuclear \uca, and $\beta$ has the \Rp, again by \autoref{Properties RP}. We claim that for every $m>1$, the restriction $\beta|_m\colon \Z_m\to\Aut(B)$ does not have the \Rp.\\
\indent Fix $m>1$, and assume that $\beta|_m$ has the \Rp. By \autoref{trivial action on K-thy}, we have $K_\ast(\beta_\zeta)=\id_{K_\ast(B)}$ for all $\zeta\in\T$. By Theorem 3.4 in \cite{Izu_RpII}, it follows that every element of $K_0(B)$ is divisible by $m$. On the other hand,
$$ (K_0(B),[1_B])\cong (K_0(A),[1_A])$$
\beqa\cong \left(\left\{\frac{a}{b}\colon a\in\Z, b=p_{k_1}\cdots p_{k_n}\colon n, k_1,\ldots,k_n\in\N, k_j\neq k_\ell \mbox{ for } j\neq \ell\right\},1\right),\eeqa
where not every element is divisible by $m$. This is a contradiction. \end{eg}

\indent We will nevertheless show that the restriction of an action of the circle with the \Rp\ to any finite cyclic subgroup again has the
\Rp\ if the algebra is separable and absorbs the universal UHF-algebra $\mathcal{Q}$. See \autoref{restriction of Rp has Rp} below.

\begin{lma}\label{unitary Rp and central sequence} Let $A$ be a separable \uca, let $n\in\N$ and let $\alpha\colon \Z_n\to\Aut(A)$ be an
action of $\Z_n$ on $A$. Regard $\Z_n\subseteq \T $ as the $n$-th roots of unitry, and let $\gamma\colon\Z_n\to \Aut(C(\T ))$ be the
restriction of the action by left translation of $\T $ on $C(\T )$. Let $\alpha_\I\colon \Z_n\to\Aut(A_\I\cap A')$ be the action on
$A_\I\cap A'$ induced by $\alpha$. Then $\alpha$ has the unitary \Rp\ \ifo there exists a unital equivariant homomorphism
\[\varphi\colon (C(\T ),\gamma) \to (A_\I\cap A',\alpha_\I).\]\end{lma}
\begin{proof} Choose an increasing sequence $(F_m)_{m\in\N}$ of finite subsets of $A$ such that $\overline{\bigcup\limits_{m\in\N}F_m}=A$. For each $m\in\N$, there exists a unitary $u_m\in A$ such that
$$\|u_ma-au_m\|<\frac{1}{m} \ \ \mbox{ and } \ \ \left\|\alpha_{j}(u_m)-e^{2\pi i j/n}u_m\right\|<\frac{1}{m}$$
for every $a\in F_m$ and for every $j\in \Z_n$. Denote by $u=\overline{(u_m)}_{m\in\N}$ the image of the sequence of unitaries $(u_m)_{m\in\N}$
in $A_\I$. Then $u$ belongs to the relative commutant of $A\in A_\I$. Consider the unital map $\varphi\colon C(\T )\to A_\I\cap A'$ given by
$\varphi(f)=f(u)$ for $f\in C(\T )$. One checks that
$$\alpha_{j}(\varphi(f))=\varphi(\gamma_{e^{2\pi ij/n}}(f))$$
for all $j\in\Z_n$ and all $f\in C(\T )$, so $\varphi$ is equivariant.\\
\indent Conversely, assume that there is an equivariant unital homomorphism
$$\varphi\colon C(\T ) \to A_\I\cap A'.$$
Let $z\in C(\T )$ be the unitary given by $z(\zt)=\zt$ for all $\zt\in \T $, and let $v=\varphi(z)$. By semiprojectivity of $C(\T )$, we
can choose a representing sequence $(v_m)_{m\in\N}\in \ell^\I(\N,A)$ consisting of unitaries. It follows that
$$\lim\limits_{m\to\I}\left\|\alpha_j(v_m)-e^{2\pi i j/n}v_m\right\|= 0=\lim\limits_{m\to\I}\|v_ma-av_m\|$$
for every $a\in A$, and this is clearly equivalent to $\alpha$ having the unitary \Rp.\end{proof}

The following result is analogous to \autoref{can replace estimate by equality}, and so is its proof.

\begin{prop}\label{can replace estimate by equality in def of uRp} Let $B$ be a separable, unital \ca, let $n\in\N$ and let $\beta\colon \Z_n\to\Aut(B)$ be an action on $B$. Then $\beta$ has the unitary Rokhlin property if and only if for every finite set $F\subseteq B$ and every $\varepsilon > 0$, there is a unitary $u\in\U(B)$ such that
\be \item $\beta_k(u)=e^{2\pi i k/n} u$ for all $k\in\Z_n$;
\item $\|ub-bu\|<\varepsilon$ for all $b\in F$.\ee\end{prop}
Similarly to what was pointed out after the statement of \autoref{can replace estimate by equality}, the definition of the unitary Rokhlin property differs in that in condition (1), one only requires $\|\beta_k(u)-e^{2\pi ik/n} u\|<\varepsilon$ for all $k\in\Z_n$.
\begin{proof} Recall that $(C(\T ), \T ,\verb"Lt")$ is equivariantly semiprojective by \autoref{rem:LtEqSj}. 
Since the quotient $\T /\Z_n$ is compact, it follows from Theorem 3.11 in \cite{PhiSorThi_eqsj} that the restriction $(C(\T ),\Z_n,\verb"Lt")$ is equivariantly semiprojective as well. The result now follows using an argument similar to the one used in the proof of \autoref{can replace estimate by equality}. The details are left to the reader. \end{proof}

\begin{prop}\label{strongly approx inner and circle Rp} Let $n\in \N$ and let $\beta\colon \Z_n\to\Aut(B)$ be an action of $\Z_n$ on
a unital separable \ca\ $B$.
\be\item The action $\beta$ has the unitary \Rp\ \ifo its dual action $\widehat{\beta}$ is strongly approximately inner.
\item The action $\beta$ is strongly approximately inner \ifo its dual action $\widehat{\beta}$ has the unitary \Rp.\ee\end{prop}
\begin{proof} Part (1). Assume that $\beta$ has the unitary \Rp. Use \autoref{unitary Rp and central sequence} to choose a unital equivariant 
\hm\ $\varphi\colon C(\T )\to B_\I\cap B'$. Denote by $u\in B_{\I}\cap B'$ the image under this \hm\ of the unitary $z\in C(\T)$ given by 
$z(\zeta)=\zt $ for $\zt\in\T$, and denote by $\lambda$ the implementing unitary representation of $\Z_n\in B\rtimes_\beta \Z_n$ for $\beta$. 
In $(B\rtimes_\beta\Z_n)_\I$, we have $u^*\lambda_j u=e^{2\pi i j/n}\lambda_j$ for all $j\in\Z_n$, and $ub=bu$ (f, $ubu^*=b$) for all $b\in B$. 
Therefore, if $\beta$ has the unitary \Rp, then $\widehat{\beta}$ is implemented by $u^*$, and thus it is strongly approximately inner. The 
converse follows from the same computation, as we have $(B\rtimes_\beta \Z_n)^{\widehat{\beta}}=B$.\\
\indent Part (2). Denote by $v$ the canonical unitary in the crossed product, and assume that $\beta$ is strongly approximately inner. Let $F\subseteq B\rtimes_{\beta}\Z_n$ be a finite subset, and let $\varepsilon>0$. Since $B$ and $v$ generate $B\rtimes_{\beta}\Z_n$, we can assume that there is a finite subset $F'\subseteq B$ such that $F=F'\cup\{v\}$. Choose $w\in\U(B)$ such that $\|\beta(w)-w\|<\varepsilon$ and $\|\beta(b)-wbw^*\|<\varepsilon$ for all $b\in F'$. Since $\beta(b)=vbv^*$ for every $b\in B$, if we let $u=w^*v$, the first of these conditions is equivalent to $\|vu-uv\|<\varepsilon$, while the second one is equivalent to $\|ub-bu\|<\varepsilon$ for all $b\in F'$. On the other hand, $\widehat{\beta}_k(u)=\widehat{\beta}_k(w^*v)=w^*(e^{2\pi i k/n}v)=e^{2\pi i k/n}u$ for $k\in\Z_n$. Thus, $u$ is the desired unitary, and $\widehat{\beta}$ has the unitary Rokhlin property.\\
\indent Conversely, assume that $\widehat{\beta}$ has the unitary Rokhlin property. Let $F'\subseteq B$ be a finite subset, and let $\varepsilon>0$. Set $F=F'\cup\{v\}$. Use \autoref{can replace estimate by equality in def of uRp} to choose $u$ in the unitary group of $A\rtimes_\beta\Z_n$ such that $\|ub-bu\|<\varepsilon$ for all $b\in F$, and $\widehat{\beta}_k(u)=e^{2\pi i k/n}u$ for all $k\in\Z_n$. Set $w=vu^*$. Then $w\in B$ since
$$\widehat{\beta}_k(w)=e^{2\pi i k/n}v\overline{e^{2\pi i k/n}}u^*=vu^*=w$$
for all $k\in\Z_n$ and $(B\rtimes_{\beta}\Z_n)^{\Z_n}=B$. On the other hand,
$$\|\beta(b)-wbw^*\|=\|vbv^*-vu^*buv^*\|=\|b-u^*bu\|=\|ub-bu\|<\varepsilon,$$
for all $b\in F$. Hence $w$ is an implementing unitary for $F'$ and $\varepsilon$, and $\beta$ is strongly approximately inner.\end{proof}

\begin{cor}\label{restriction of Rokhlin, dual is str approx inner} Let $\alpha\colon\T\to\Aut(A)$ be an action with the \Rp, and let $n\in\N$. Then $\widehat{\alpha|_n}\colon\Z_n\to\Aut(A\rtimes_{\alpha|_n}\Z_n)$ is strongly approximately inner.\end{cor}
\begin{proof} The restriction $\alpha|_n$ has the unitary \Rp\ by \autoref{restriction has uRp}, and by part (1) of
\autoref{strongly approx inner and circle Rp}, its dual $\widehat{\alpha|_n}$ is strongly approximately inner.\end{proof}

The next ingredient needed is showing that crossed products by restrictions of Rokhlin actions of compact groups preserve absorption of
strongly self-absorbing \ca s. For Rokhlin actions, this was shown by Hirshberg and Winter in \cite{HirWin_Rp}. The more general
statement is proved using similar ideas.

\begin{prop}\label{sequence unital embeddings fixed point restriction} Let $A$ be a separable unital \ca, let $G$ be a compact Hausdorff second countable group, and let $\af\colon G \to \Aut(A)$ be an action satisfying the Rokhlin property. Let $H$ be a closed subgroup of $G$. If $B$ is a unital separable \ca\ which admits a central sequence of unital homomorphisms into $A$, then $B$ admits a unital homomorphism into the fixed point subalgebra of $\alpha|_H\in A_\I\cap A'$.\end{prop}
\begin{proof} Notice that $\left(A_\I\cap A'\right)^\alpha$ is a subalgebra of $\left(A_\I\cap A'\right)^{\alpha|_H}$. The result now follows from Theorem 3.3 in \cite{HirWin_Rp}.\end{proof}

\begin{rem} In the proposition above, if $B$ is moreover assumed to be simple, for example if it is strongly self-absorbing, it follows that the unital homomorphism obtained is actually an embedding, since it is not zero.\end{rem}

Recall the following result by Hirshberg and Winter.

\begin{lma}\label{embedding fixed by the action} (Lemma 2.3 of \cite{HirWin_Rp}.) Let $A$ and $\D$ be unital separable \ca s. Let $G$ be a
compact group and let $\alpha\colon G\to\Aut(A)$ be a continuous action. If there is a unital homomorphism $\D\to (A_\I\cap A')^G$, then
there is a unital homomorphism
\[\D \to \left(M(A\rtimes_\alpha G)\right)_\I\cap (A\rtimes_\alpha G)'.\] \end{lma}

\begin{thm}\label{crossed product of restriction absorbs D} Let $\D$ be a strongly self-absorbing \ca, let $A$ be a $\D$-absorbing, separable \uca,
and let $\alpha\colon\T\to\Aut(A)$ be an action with the \Rp. Then, for every $n\in\N$, the crossed product $A\rtimes_{\alpha|_n}\Z_n$ is a unital
separable $\D$-absorbing \ca.\end{thm}
\begin{proof} By Theorem 7.2.2 in \cite{Ror_BookClassif}, there exists a unital embedding of $\D$ into $A_\I\cap A'$, which is equivalent to
the existence of a central sequence of unital embeddings of $\D$ into $A$. Use \autoref{sequence unital embeddings fixed point restriction}
to obtain a unital homomorphism of $\D$ into the fixed point subalgebra of $\alpha|_{\Z_n}\in A_\I\cap A'$. It follows that this homomorphism is
actually an embedding, since it is not zero and $\D$ is simple, by Theorem 1.6 in \cite{TomWin_SSA}. Lemma 2.3 in \cite{HirWin_Rp}
(here reproduced as \autoref{embedding fixed by the action})
provides us with a unital embedding of $\D$ into $(A\rtimes_\alpha \Z_n)_\I\cap (A\rtimes_\alpha \Z_n)'$, which again by Theorem 7.2.2 in
\cite{Ror_BookClassif} is equivalent to $A\rtimes_\alpha \Z_n$ being $\D$-absorbing, since $\D$ is strongly self-absorbing. \end{proof}

The following is the main theorem of this section.

\begin{thm}\label{characterize restriction has Rp} Let $A$ be a separable \uca, let $n\in\N$ and let $\alpha\colon\T\to\Aut(A)$ be an action with the \Rp. Suppose that $A$ absorbs $M_{n^\I}$. Then ${\alpha|_n}$ has the \Rp.\end{thm}
\begin{proof} By \autoref{approx repres and Rp}, it is enough to show that $\widehat{\alpha|_n}\colon\Z_n\to\Aut(A\rtimes_{\alpha|_n}\Z_n)$ is approximately representable. Recall that by \autoref{restriction of Rokhlin, dual is str approx inner}, the action $\widehat{\alpha|_n}$ is strongly approximately inner. In view of Lemma 3.10 in \cite{Izu_RpI}, to show that it is in fact approximately representable, it is enough to show that there is a unital map
$$M_n\to \left((A\rtimes_{\alpha|_n}\Z_n)^{\widehat{{\alpha|_n}}}\right)_\I\cap (A\rtimes_{\alpha|_n}\Z_n)',$$
where the relative commutant is taken in $(A\rtimes_{\alpha|_n}\Z_n)_\I$.

\vspace{0.3cm}

\textbf{\underline{Claim:}} $\left((A\rtimes_{\alpha|_n}\Z_n)^{\widehat{{\alpha|_n}}}\right)_\I\cap (A\rtimes_{\alpha|_n}\Z_n)'=(A_\I \cap A')^{(\alpha|_n)_\I}$.\\
\indent Since $(A\rtimes_{\alpha|_n}\Z_n)^{\widehat{\alpha|_n}}=A$, we have
\begin{align*}\left((A\rtimes_{\alpha|_n}\Z_n)^{\widehat{{\alpha|_n}}}\right)_\I\cap (A\rtimes_{\alpha|_n}\Z_n)'=A_\I\cap (A\rtimes_{\alpha|_n}\Z_n)'\\
=\left\{\overline{(a_m)}_{m\in\N}\in A_\I\colon \lim\limits_{m\to\I}\|a_mx-xa_m\|= 0 \mbox{ for all } x\in A\rtimes_{\alpha|_n}\Z_n\right\}. \end{align*}
\indent Let $v$ be the canonical unitary in $A\rtimes_{\alpha|_n}\Z_n$ that implements $\alpha|_n$ in the crossed product. Then for a bounded sequence
$(a_m)_{m\in\N}\in A$, the condition $\lim\limits_{m\to\I}\|a_mx-xa_m\|= 0$ for all $x\in A\rtimes_{\alpha|_n}\Z_n$ is equivalent to
$\lim\limits_{m\to\I}\|a_ma-aa_m\|= 0$ for all $a\in A$ and $\lim\limits_{m\to\I}\|a_mv-va_m\|= 0$. The above set is therefore equal to
$$\left\{\overline{(a_m)}_{m\in\N}\in A_\I\colon
\begin{aligned}
& \lim\limits_{m\to\I}\|a_ma-aa_m\|= 0 \mbox{ for all } a\in A \mbox{ and }\\
& \ \ \ \ \ \ \ \ \lim\limits_{m\to\I}\|(\alpha|_n)(a_m)-a_m\|= 0
\end{aligned}
 \right\},$$
which is precisely the same as the subset of $A_\I\cap A'$ that is fixed under the action on $A_\I\cap A'$ induced by $\alpha|_n$. This proves the claim.

\vspace{0.3cm}

Since $A$ absorbs the UHF-algebra $M_{n^\I}$, it follows that there is a unital embedding $\iota\colon M_n\to A_\I\cap A'$. By \autoref{sequence unital embeddings fixed point restriction}, there is a unital homomorphism $M_n \to (A_\I\cap A')^{(\alpha|_n)_\I}$.
Using the claim above, we conclude that there is a unital homomorphism
$$M_n \to \left((A\rtimes_{\alpha|_n}\Z_n)^{\widehat{{\alpha|_n}}}\right)_\I\cap (A\rtimes_{\alpha|_n}\Z_n)'.$$
This homomorphism is necessarily an embedding, since it is not zero. Apply Lemma 3.10 in \cite{Izu_RpI} to the action $\widehat{\alpha|_n}\colon\Z_n\to \Aut(A\rtimes_{\alpha|_n}\Z_n)$ to conclude that $\widehat{\alpha|_n}$ is approximately representable, and hence that $\alpha|_n$ has the \Rp, thanks to \autoref{approx repres and Rp}. \end{proof}

The following partial converse to \autoref{characterize restriction has Rp} holds. The same result is likely to be true for a larger class of \ca s.

Recall that if $A$ is a stably finite \ca, then its Murray-von Neumann semigroup
$V(A)$ can be naturally identified with the subsemigroup of its Cuntz semigroup $\mathrm{Cu}(A)$ consisting of
compact elements. Additionally, if $A$ has real rank zero, then every element in $\mathrm{Cu}(A)$ is the
supremum of an increasing sequence in $V(A)\subseteq \mathrm{Cu}(A)$.

\begin{thm}\label{converse to restriction has Rp if algebra absorbs UHF}
Let $A$ belong to one of the following classes of \ca s:
\be\item Unital Kirchberg algebras satisfying the UCT;
\item Simple, nuclear, separable \uca s with tracial rank zero satisfying the UCT;
\item Unital real rank zero direct limits of one-dimensional noncommutative CW-complexes with trivial
$K_1$-group (this includes all AF-algebras).\ee
Let $\alpha\colon\T\to\Aut(A)$ be a continuous action and let $n\in\N$. If $\alpha|_n$ has the \Rp, then $A$ absorbs $M_{n^\I}$. \end{thm}
\begin{proof} Assume that $\alpha|_n$ has the \Rp. Since $\alpha$ induces the trivial action on $K$-theory by \autoref{trivial action on K-thy},
so does $\alpha|_n$. For algebras in the first two classes, the result follows from Theorem 3.4 in
\cite{Izu_RpII} and Theorem 3.5 in \cite{Izu_RpII}, respectively. For algebras in the third class, it follows
that $\alpha|_n$ acts trivially on the Murray-von Neumann semigroup $V(A)$. Since every element in $\mathrm{Cu}(A)$
is the supremum of an increasing sequence in $V(A)$, it follows that $\alpha|_n$ acts trivially on $\mathrm{Cu}(A)$ as
well. The result now follows from Theorem 4.2 in \cite{GarSan_RokConstrI}.\end{proof}

Denote by $\mathcal{Q}$ the universal UHF-algebra, that is, the unique, up to isomorphism, UHF-algebra with $K$-theory
$$\left(K_0(\mathcal{Q}),\left[1_{\mathcal{Q}}\right]\right)\cong (\Q,1).$$
It is well-known that $\mathcal{Q}\otimes M_{n^\I}\cong \mathcal{Q}$ for all $n\in \N$, and that $\mathcal{Q}\otimes \Ot\cong \Ot$.

\begin{cor}\label{restriction of Rp has Rp} Let $A$ be a separable, $\mathcal{Q}$-absorbing \uca, let $\alpha\colon\T\to\Aut(A)$ be an action
with the \Rp\ and let $n\in\N$. Then ${\alpha|_n}$ has the \Rp. In particular, restrictions to finite subgroups of circle actions with the \Rp\ on separable, unital
$\Ot$-absorbing \ca s, again have the \Rp. \end{cor}

We finish this work by showing that the Rokhlin property for a circle action cannot in general be
determined just by looking at its restrictions to finite subgroups.

\begin{eg}\label{all restrictions have Rp, but action does not} There are a \uca\ $A$ and a circle action on $A$ such that its restriction
to every proper subgroup has the \Rp, but the action itself does not. \\
\indent Let $A$ be the universal UHF-algebra, that is, $A=\varinjlim(M_{n!},\iota_n)$ where $\iota_n\colon M_{n!}\to M_{(n+1)!}$ is
given by $\iota_n(a)=\diag(a,\ldots,a)$ for all $a\in M_{n!}$. For every $n\in\N$, let $\alpha^{(n)}\colon \T\to\Aut(M_{n!})$ be given by
$$\alpha^{(n)}_\zeta=\Ad(\diag(1,\zt,\ldots,\zt^{n!-1}))$$
for all $\zeta\in\T$. Then $\iota_n\circ\alpha^{(n)}_\zeta=\alpha^{(n+1)}_\zeta\circ\iota_n$ for all $n\in\N$ and all $\zeta\in\T$,
and hence there is a direct limit action $\alpha=\varinjlim\alpha^{(n)}$ of $\T$ on $A$. This action does not have the \Rp\ by \autoref{no direct limit actions with the Rp on AF-algebras}.\\
\indent On the other hand, we claim that given $m\in\N$, the restriction $\alpha|_m\colon \Z_m\to\Aut(A)$ has the \Rp. So fix $m\in\N$. Then $\alpha|_m$ is the direct limit of the actions $\left(\alpha^{(n)}|_m\right)_{n\in\N}$, whose generating automorphisms are
$$\alpha^{(n)}_{e^{2\pi i/m}}=\Ad(\diag(1,e^{2\pi i/m},\ldots,e^{2\pi i(n!-1)/m})).$$
Let $F\subseteq A$ be a finite subset and let $\varepsilon>0$. Write $F=\{a_1,\ldots,a_N\}$. Since $\bigcup\limits_{n\in\N}M_{n!}$ is dense in $A$, there are $k\in\N$ and a finite subset $F'=\{b_1,\ldots,b_N\}\subseteq M_{k!}$ such that $\|a_j-b_j\|<\frac{\ep}{2}$ for all $j=1,\ldots,N$. \\
\indent Let $n\geq \max\{k,m\}$. Then the $\Z_m$-action $\alpha^{(n)}|_m$ on $M_{n!}$ is generated by the automorphism
$$\alpha^{(n)}_{e^{2\pi i/m}}=\Ad(1,e^{2\pi i/m},\ldots,e^{2\pi i(m-1)/m},\ldots,1,e^{2\pi i/m},\ldots,e^{2\pi i(m-1)/m}).$$
(There are $n!/m$ repetitions.) Denote by $e_0$ the projection
$$1_{M_{(n-1)!}}\otimes \left(
               \begin{array}{ccccccccccc}
                 \frac{1}{m} & \frac{1}{m} & \cdots & \frac{1}{m} & 0 &\cdots & 0 & 0 & 0 & \cdots & 0 \\
                 \frac{1}{m} & \frac{1}{m} & \cdots & \frac{1}{m} & 0 &\cdots & 0 & 0 & 0 & \cdots & 0 \\
                 \vdots & \vdots & \ddots & \vdots &\vdots &\cdots & \vdots & \vdots & \vdots & \cdots & \vdots \\
                 \frac{1}{m} & \frac{1}{m} & \cdots & \frac{1}{m}& 0 & \cdots & 0 & 0 & 0 & \cdots & 0 \\
                 0 & 0 & \cdots & 0 & \frac{1}{m} & \cdots & 0 & 0 & 0 & \cdots & 0\\
                 \vdots & \vdots &\cdots &\vdots & \vdots &\ddots & \vdots &\vdots &\vdots& \cdots& \vdots \\
                 0 & 0 &\cdots & 0 & 0 &\cdots & \frac{1}{m} & 0 & 0 & \cdots & 0\\
                 0 & 0 &\cdots & 0 & 0 &\cdots & 0 & \frac{1}{m} & \frac{1}{m} & \cdots & \frac{1}{m}\\
                 0 & 0 &\cdots & 0 & 0 &\cdots & 0 & \frac{1}{m} & \frac{1}{m} & \cdots & \frac{1}{m}\\
                 \vdots & \vdots &\cdots & \vdots & \vdots &\cdots &\vdots & \vdots & \vdots& \ddots & \vdots\\
                 0 & 0 &\cdots & 0 & 0  &\cdots & 0 & \frac{1}{m} & \frac{1}{m} & \cdots & \frac{1}{m}\\
                 \end{array}
             \right)$$
in $M_{n!}\subseteq A$, and for $j=1,\ldots, m-1$, set $e_j =\alpha^{(n)}_{e^{2\pi ij/m}}(e_0)\in A$. One checks that
$e_0,\ldots,e_{m-1}$ are orthogonal projections with $\sum\limits_{j=0}^{m-1} e_j=1$, and moreover that $\alpha^{(n)}_{e^{2\pi i/m}}(e_{m-1})=e_0$. \\
\indent By construction, these projections are cyclically permuted by the action $\alpha|_m$ and they sum up to one, so we only need to check that they almost commute with the given finite set. The projections $e_0,\ldots,e_{m-1}$ exactly commute with the elements of $F'$. Thus, if $k\in\{1,\ldots,N\}$ and $j\in\{0,\ldots,m-1\}$, then
\begin{align*} \|a_ke_j-e_ja_k\|&\leq \|a_ke_j-b_ke_j\|+\|b_ke_j-e_jb_k\|+\|e_jb_k-e_ja_k\|\\
&<\frac{\ep}{2}+\frac{\ep}{2}=\varepsilon,\end{align*}
and hence $\alpha|_m$ has the \Rp. \end{eg}

The phenomenon exhibited in the example above is not special to UHF-algebras:

\begin{eg}\label{all restrictions have Rp, but action does not on pi}
If $A$ and $\alpha$ are as in \autoref{all restrictions have Rp, but action does not}, set 
$B=A\otimes\OI$ and let $\beta\colon\T\to\Aut(B)$ be given by $\beta_\zeta=\alpha_\zeta\otimes\id_{\OI}$ 
for all $\zeta\in\T$. Then $B$ is a unital Kirchberg algebra satisfying the UCT. We claim that the action $\beta$ does 
not have the \Rp. To see this, observe first that the fixed point algebra $A^\alpha$ can be written as an inductive
limit $A^\alpha=\varinjlim M_{n!}^{\alpha^{(n)}}$, and that 
\[M_{n!}^{\alpha^{(n)}}=\{a\in M_{n!}\colon a \mbox{ commutes with } \diag(1,\zeta,\ldots,\zeta^{n!-1}) \ \forall \ \zeta\in\T\}=\C^{n!}\subseteq M_{n!},\]
the last embedding being as diagonal matrices. In particular, the fixed point algebra $B^\beta=A^\alpha\otimes \OI$ is purely infinite
but not simple, and thus the Rokhlin property for $\beta$ would contradict Theorem~6.3 in~\cite{Gar_Kir1}. 

On the other hand, for every $m\in\N$, the restriction $\beta|_m\colon \Z_m\to\Aut(B)$ has the \Rp\ by part~(i) of Theorem~2
in~\cite{Santiago}, being a
tensor product of an action with the Rokhlin property (namely $\alpha|_m$) and another action.
\end{eg}


\end{document}